\documentclass[12pt]{amsart}

\usepackage[margin=1.00in]{geometry}
\usepackage{amscd,amssymb, amsmath, setspace}
\usepackage{graphicx}
\usepackage{url}
\usepackage{hyperref}%make all the references and links clickable

\theoremstyle{plain}
\newtheorem{theorem}{Theorem}[section]
\newtheorem{prop}[theorem]{Proposition}

\newtheorem{cor}[theorem]{Corollary} 

\newtheorem{lemma}[theorem]{Lemma}

\theoremstyle{definition}

\newtheorem*{ex*}{Example}

\newcommand\sO{{\mathcal O}}
\newcommand\sK{{\mathcal K}}
\newcommand\sH{{\mathcal H}}

\newcommand\sP{{\mathcal P}}

\newcommand{\codim}{\rm codim \,}
\newcommand{\ddim}{\rm dim \,}

\newcommand\pp{{\mathbf{P}}}
\newcommand\zz{{\mathbf{Z}}}

\newcommand\cc{{\mathbf{C}}}

\newcommand\llll{{\mathbf{L}}}

\newcommand\gggg{{\mathbf{G}}}

\title[Inequalities for the Hodge numbers]{Inequalities for the Hodge numbers of irregular compact K\"{a}hler manifolds}

\author{Luigi Lombardi}
\address{Department of Mathematics, Statistics, and Computer Science\\
University of Illinois at Chicago, 851 S. Morgan Street, Chicago, IL, 60607}
\email{\url{lombardi@math.uic.edu}}

\begin{document}

\begin{abstract}
Based on work of R. Lazarsfeld and M. Popa, we use the derivative complex associated to the bundle of the holomorphic 
$p$-forms to provide inequalities for all the Hodge numbers of a special class of irregular compact
K\"{a}hler manifolds. For $3$-folds and $4$-folds we give an asymptotic bound for all the Hodge numbers in terms 
of the irregularity. As a byproduct, via the BGG correspondence, we also bound the regularity of the exterior 
cohomology modules of bundles of holomorphic $p$-forms.
\end{abstract}

\maketitle

\section{Introduction}

Given an irregular compact K\"{a}hler manifold $X$, a problem one tries to understand is under which hypotheses there exist 
relations between its various Hodge numbers $h^{p,q}(X)$. Along these lines, one can ask whether there exist
formulas for the $h^{p,q}(X)$'s in terms of the fundamental invariant $q(X)=h^{1,0}(X)$, the \emph{irregularity} of $X$. 
% A first result in this direction was given by G. Castelnuovo and M. De Franchis 
% which asserts that if $X$ is a surface that does not admit any
% surjective morphisms with connected fibers onto a smooth curve of genus at least two (i.e. \emph{irrational pencils of genus} $\geq 2$), then
% \begin{equation}\label{cdf}
% h^{0,2}(X)\geq 2q(X)-3 
% \end{equation}
% where $q(X)$ is the irregularity of $X$. This inequality is called \emph{Castelnuovo-De Franchis inequality} and its
% proof relies on an argument involving holomorphic one-forms on $X$, which is the content of the Castelnuovo-De Franchis Lemma (cf. \cite{Ca},
% \cite{DF}, \cite{BHPV} IV.5.4). 
A classical result in this direction is the Castelnuovo-De Franchis inequality $h^{0,2}(X)\geq 2q(X)-3$ 
(see \cite{BHPV} IV.5.2), which holds for surfaces that do not carry any fibrations onto a smooth curve of genus $g\geq 2$.
This was generalized by F. Catanese to higher dimensional manifolds as follows: if a manifold $X$ does not admit
any higher irrational pencil\footnote{A \emph{higher irrational pencil} is a surjective map with connected fibers $f:X\longrightarrow
Y$ having the property that any resolution singularities of $Y$ is of maximal Albanese dimension and with non-surjective Albanese map.} 
then $h^{0,k}(X)\geq k(q(X)-k)+1$ for all $k$ (\emph{cf.} \cite{CAT}), by means of sophisticated arguments involving the exterior algebra
of holomorphic forms. 
% The analogous in higher dimension of irrational pencils of genus $\geq 2$ for surfaces, are the \emph{higher irregular pencils}, 
% i.e. surjective morphisms $f:X\longrightarrow Y$
% with connected fibers onto a lower dimensional normal variety $Y$ having the property 
% that any smooth model of $Y$ has maximal Albanese dimension and non surjective
% Albanese map. In fact in \cite{CAT} F. Catanese, by generalizing the Castelnuovo-De Franchis Lemma for higher dimensional varieties, proved 
% that if $X$ does not carry any higher irrational pencils then
% $$h^{0,k}(X)\geq k(q(X)-k)+1$$ for any $k$, in particular this includes the Castelnuovo-De Franchis inequality (cf. \cite{CAT}, \cite{LP} 
% Remark 3.3). An improvement of the previous inequalities 
% for $k=2$ and for $q(X)\leq 2\ddim X-1$ was given by A. Causin and G. Pirola in \cite{CP} $$h^{0,2}(X)\geq \binom{q(X)}{2}.$$
Another generalization of the Castelnuovo-De Franchis inequality is provided by the work of G. Pareschi and M. Popa. 
% For the case of surfaces, the inequality is equivalent to the inequality $\chi (\omega_X)\geq q(X)-\ddim X$,
Theorem A in \cite{GSV} states that if $X$ is of maximal Albanese dimension and does not carry any higher irrational pencil then
$\chi (\omega_X)\geq q(X)-\ddim X.$ Their inequality is deduced using Generic Vanishing Theory for irregular varieties and
the Evans-Griffith Syzygy Theorem.

New techniques for the study of this problem  
were introduced recently by R. Lazarsfeld and M. Popa in \cite{LP}. Their approach 
relies on the study of a global version of the \emph{derivative complex associated to the structure sheaf} $\sO_X$ and
on the theory of vector bundles on projective spaces. Their inequalities mainly involve Hodge numbers
of type $h^{0,k}(X)$.

In this paper we extend the methods of \cite{LP} to the bundles of holomorphic $p$-forms $\Omega_X^{p}$. In this way we 
get inequalities for all the Hodge numbers of a special class of irregular compact K\"{a}hler manifolds. The main idea 
behind the inequalities of \cite{LP} and the ones of this paper goes as follows.
Via cup product any element $0\neq v\in H^1(X,\sO_X)$ defines a complex of vector spaces
\begin{eqnarray}\label{derivative}
0\longrightarrow H^0(X,\Omega_X^{p})\stackrel{\cup v}{\longrightarrow}H^1(X,\Omega_X^{p})
\stackrel{\cup v}{\longrightarrow}\ldots \stackrel{\cup v}{\longrightarrow} H^d(X,\Omega_X^{p})\longrightarrow 0.
\end{eqnarray}
Denoting by $\pp=\pp_{\rm sub}(H^1(X,\sO_X))$ the projective space of one dimensional linear subspaces of $H^1(X,\sO_X)$, 
we can arrange all of these, as $v$ varies, into a complex of locally free sheaves on $\pp$:
\begin{eqnarray}\label{BGG-intr}
0\longrightarrow \sO_{\pp}(-d)\otimes H^0(X,\Omega_X^{p})
\longrightarrow \sO_{\pp}(-d+1)\otimes H^1(X,\Omega_X^{p})\longrightarrow \ldots
\end{eqnarray}
$$\ldots \longrightarrow \sO_{\pp}(-1)\otimes H^{d-1}(X,\Omega_X^{p})\longrightarrow \sO_{\pp}\otimes H^d(X,\Omega_X^{p})\longrightarrow 0,$$
whose fiber at a point $[v]\in \pp$ is the complex \eqref{derivative}.
The complex (\ref{derivative}) is known as the \emph{derivative complex associated to the bundle} $\Omega_X^{p}$ \emph{with respect to the 
vector} $v$. It was introduced for the first time by
M. Green and R. Lazarsfeld in \cite{GL1} and \cite{GL2} in their study of Generic Vanishing Theorems for irregular compact K\"{a}hler manifolds. 
At this point it is important to study the exactness of the complex (\ref{BGG-intr}).
In \cite{LP} the case $p=0$ is analyzed. For this case we have that if $X$ does not carry any irregular fibrations, i.e. morphisms
$f:X\longrightarrow Y$ with connected fibers having the property that any smooth model of $Y$ is of maximal Albanese dimension, 
then the complex \eqref{BGG-intr} is everywhere exact except possibly at the last 
term and furthermore all the involved maps are of constant rank. 
In general, for $p>0$, in Proposition \ref{exact} we show that the exactness of the complex \eqref{BGG-intr} depends on 
the non-negative integer $m(X)$, the
least codimension of the zero-locus of a non-zero holomorphic one-form, i.e.
$$m(X)={\rm min}\{{\rm codim}\;Z(\omega)\;|\;0\neq \omega\in H^0(X,\Omega_X^1)\},$$ with the convention
$m(X)=\infty$ if every non-zero holomorphic $1$-form is everywhere non-vanishing. 
For instance, it was observed in \cite{GL1} Remark on p. 405 
and in \cite{LA} Proposition 6.3.10,   
examples of varieties 
with $m(X)=\ddim X$ are the smooth subvarieties of an abelian variety with ample normal bundle
(hence all smooth subvarieties of a simple abelian variety).
% This can be done since the derivative complex associated to the structure sheaf 
% can be identified with the Bernstein-Gel'fand-Gel'fand 
% correspondence which is related with the regularity of the cohomology module $\oplus_i H^i(X,\omega_X)$.
In particular we show that if $m(X)>p$ then the complex \eqref{BGG-intr} is exact at the first ($m(X)-p$)-steps from the left with
the first $m(X)-p$ maps being of constant rank. This is enough to ensure that the cokernel of the map
$$\sO_{\pp}(m(X)-d-p-1)\otimes H^{m(X)-p-1}(X,\Omega_X^{p})\longrightarrow \sO_{\pp}(m(X)-d-p)\otimes H^{m(X)-p}(X,\Omega_X^{p})$$
is a locally free sheaf. This in turn leads to inequalities for the Hodge numbers by the Evans-Griffith Theorem and by the fact that the 
Chern classes of a globally generated locally free sheaf are non-negative.

We turn to a more detailed presentation of our results. To simplify notation we only present the case $m(X)=\ddim X$ and we 
refer to Theorem \ref{schur} and Theorem \ref{schur2} for general statements where all possible values of $m(X)$ are considered. 
Fix an integer $1\leq p \leq d$ and for any $1\leq i\leq q-1$ define $\gamma_i(X,\Omega_X^{p})$ to be
the coefficient of $t^i$ in the formal power series:
$$\gamma(X,\Omega_X^{p};t)\stackrel{def}{=} \prod_{j=1}^{p}(1-jt)^{(-1)^j h^{p,d-p+j}}\in \zz[[t]],$$ 
where $h^{i,j}=h^{i,j}(X)$ are the Hodge numbers of $X$. 
% Moreover we also consider 
% the following partial holomorphic Euler characteristic of the bundle $\Omega_X^{d-p}$ starting from step $p$ up to $d$:
% $$\chi^{\geq p}(\Omega_X^{d-p})\stackrel{def}{=}\sum_{j=p}^{d} (-1)^{j+p} h^{d-p,j}.$$ Then we obtain the following 
\begin{theorem}\label{thm-intr2}
Let $X$ be a compact K\"{a}hler manifold of dimension $d$, irregularity $q\geq 1$ and with $m(X)=d$.  Then
\begin{enumerate} 
  \item Any Schur polynomial of weight $\leq q-1$ in the $\gamma_i(X,\Omega_X^{p})$ is non-negative.
In particular $$\gamma_i(X,\Omega_X^{p})\geq 0.$$\\
\item If $q>{\rm max}\;\{p,d-p-1\}$ then $$\sum_{j=d-p}^{d} (-1)^{d-p+j} h^{p,j}\geq q-p.$$
 \end{enumerate}
\end{theorem}
% In this way for surfaces with $m(X)=2$ and irregularity $q>2$ we recover the well known inequalities true 
% for surfaces without irrational
% pencils (i.e. with no non-constant morphism to a curve of genus at least 2)
% $$h^{0,2}\geq 2q-1,\quad h^{1,1}\geq 2q-1,$$ the first one is the Castelnuovo-De Franchis inequality and the second can be found in
% \cite{BHPV} IV.5.4. (in fact the inequality $h^{0,2}\geq 2q-3$ holds even in the case $q=2$, see Theorem \ref{schur2} (iii)).
For instance, the $\gamma_1(X,\Omega_X^{p})$'s are non-negative linear polynomials in the variables $h^{p,j}$, and in the case of surfaces 
the inequalities above
become $h^{0,2}\geq 2q-3$ and $h^{1,1}\geq 2q-1$, well-known inequalities true for surfaces which do not admit any irregular pencils of genus 
$\geq 2$ (\emph{cf.} \cite{BHPV} IV.5.4).
% well-known inequalities which hold when the surface does not carry any irrational pencils 
% of genus $\geq 2$ (i.e. no surjective morphisms with connected fibers onto a smooth curve of genus $\geq 2$, cf. 
% \cite{BHPV} IV.5.4).
In higher dimension the polynomials $\gamma_i(X,\Omega_X^{p})$, of degree $i$ in the variables $h^{p,j}$,
give new inequalities involving Hodge numbers of $X$. In addition to the methods for $\sO_X$, for 
$p>0$ Serre Duality offers a way to see until which step the complex \eqref{BGG-intr} is exact 
\emph{counting from the right}. 
This trick leads to further inequalities for the Hodge numbers and, in the special case when $m(X)=\ddim X$ and $q(X)>\ddim X\geq 2$, 
also to get a bound on the Euler characteristic for the bundles $\Omega_X^p$:
$$|\chi (\Omega_X^1)|\geq 2 \quad \mbox{ and } \quad |\chi(\Omega_X^p)|\geq 1\quad \mbox{ for }\quad p=2,\ldots, \ddim X-2$$ 
(\emph{cf.} Corollary \ref{cor}).
In section IV we list 
the inequalities coming from Theorem \ref{thm-intr2} for manifolds of dimension three, four and five. Finally 
for threefolds and fourfolds with $m(X)=\ddim X$
we are able to give asymptotic bounds for all the Hodge numbers in terms of the irregularity $q$.
In the case of threefolds we obtain
$$h^{0,2}\succeq 4q,\quad h^{0,3}\succeq 4q, \quad h^{1,1}\succeq 2q+\sqrt{2q},\quad h^{1,2}\succeq 5q+\sqrt{2q}$$
and for the case of fourfolds we get
$$h^{0,2}\succeq 4q, \quad h^{0,3}\succeq 5q+\sqrt{2q},\quad h^{0,4}\succeq 4q$$
$$h^{1,1}\succeq 2q,\quad h^{1,2}\succeq 8q+2\sqrt{2q},\quad h^{1,3}\succeq 12q+3\sqrt{2q},\quad h^{2,2}\succeq 8q+4\sqrt{2q}.$$
Asymptotic inequalities for Hodge numbers of type $h^{0,j}$ were already established in \cite{LP} for manifolds which do not carry any irregular
fibrations. 
% In fact under this point of view manifolds with $m(X)=\dim (X)$ and manifolds carrying no irregular fibrations behave likely and 
% so far we do not know if there is any relation between them.

Setting $E=\bigwedge ^*H^1(X,\sO_X)$ for the graded exterior algebra over $H^1(X,\sO_X)$, in 
the last section, we use the Bernstein-Gel'fand-Gel'fand (BGG) correspondence and Generic Vanishing Theorems for 
bundles of holomorphic $p$-forms to bound the \emph{regularity} of the $E$-modules
$\bigoplus_i H^i(X,\Omega_X^{p})$ for any value of $p$. 
% This can be done because 
% the complex \eqref{BGG-intr} can be identified, as in \cite{LP}, with a 
% BGG complex associated to the $E$-module $\bigoplus_i H^i(X,\Omega_X^{p})$. 
We refer to Section V for 
the definition of regularity for finitely generated graded modules over an exterior algebra and for references about the BGG correspondence 
and the Generic Vanishing Theorems used.
The case $p=\ddim X$ has been studied in \cite{LP} Theorem B. If we denote by $k$ for the dimension of the 
general fiber of the Albanese map ${\rm alb}_X:X\longrightarrow {\rm Alb}(X)$, then the $E$-module
$\bigoplus_i H^i(X,\omega_X)$ is $k$-regular but not $(k-1)$-regular. In general, for all the others values of $p$, we are not able to determine 
the regularity of the $E$-module $\bigoplus_i H^i(X,\Omega_X^p)$ but only to give a bound in terms of the minimal and 
maximal dimensions of the fibers of 
the Albanese map ${\rm alb}_X:X\longrightarrow {\rm Alb}(X)$. 
\begin{theorem}\label{thm-intr1}
Let $X$ be a compact K\"{a}hler manifold of dimension $d$ and irregularity $q\geq 1$. 
Let $k$ be the dimension of the general fiber of ${\rm alb}_X:X\longrightarrow {\rm Alb}(X)$ and 
$f$ be the maximal dimension of a fiber of ${\rm alb}_X$. Let $0\leq p\leq d$ be an integer and set 
$l={\rm max}\{k, f-1\}$. 
If $p>l$ then the $E$-module $\bigoplus_i H^i(X,\Omega_X^p)$ is $(d-p+l)$-regular.\\
\end{theorem}

\textbf{Acknowledgements.} I am very grateful to Professor M. Popa for drawing my attention to this problem and for many helpful discussions. 
I also want to thank Professors L. Ein, R. Lazarsfeld and C. Schnell for useful conversations.

\section{Exactness of the complex $\underline{\llll}_X^p$}
Let $X$ be a compact K\"{a}hler manifold of dimension $d$. The \emph{irregularity} of $X$ is the non-negative integer 
$q(X)\stackrel{def}{=}h^1(X,\sO_X)$. The manifold $X$ is said to be \emph{irregular} if $q(X)>0$.
We aim to study the exactness of the complex \eqref{BGG-intr} in terms of the non-negative integer 
$$m=m(X):={\rm min} \{\codim Z(\omega)\; | \; 0\neq \omega\in H^0(X,\Omega_X^1)\},$$
with the convention
$m(X)=\infty$ if every non-zero holomorphic $1$-form is nowhere vanishing. Refer also to Proposition \ref{alb} for another study of the 
exactness in terms of different invariants.
\begin{prop}\label{exact}
Let $X$ be an irregular compact K\"{a}hler manifold of dimension $d$ and $0\leq p\leq d$ be an integer.
\begin{enumerate}
\item If $p<m\leq d$ then the complex $\underline{\mathbf{L}}_X^p$ is exact at the first $m-p$ steps from the left, and the 
first $m-p$ maps are of constant rank.\\
\item If $d-p<m\leq d$ then the complex $\underline{\mathbf{L}}_X^p$ is exact at the first $m-d+p$ steps from the right, and the 
last $m-d+p$ maps are of constant rank.\\
\item If $m=\infty$ then the whole complex $\underline{\mathbf{L}}_X^p$ is exact and all the involved maps are of constant rank.\\
\end{enumerate}
\end{prop}

\begin{proof}
Under the Hodge conjugate-linear isomorphism $H^i(X,\Omega_X^{j})\cong H^j(X,\Omega_X^i)$, the fiber at a point $[v]\in \pp$ of the
complex $\underline{\llll}_X^p$ is identified with the complex
\begin{equation}\label{fiber}
0\longrightarrow H^p(X,\sO_X)\stackrel{\wedge \omega}{\longrightarrow}H^p(X,\Omega_X^1)\stackrel{\wedge \omega}{\longrightarrow}\ldots 
\stackrel{\wedge \omega}{\longrightarrow}
H^p(X,\Omega_X^{d})\longrightarrow 0,
\end{equation}
where $\omega\in H^0(X,\Omega_X^1)$ is the holomorphic $1$-form conjugate to $v\in H^1(X,\sO_X)$.
For every non-zero holomorphic one-form $\omega$ the complex \eqref{fiber} is exact 
at the first $m-p$ steps from the left by \cite{GL1} Proposition 3.4. 
Hence the complex $\underline{\llll}_X^p$ is itself exact since exactness can be checked at the level of fibers. 
This also shows that the first $m-p$ maps of the complex $\underline{\llll}_X^p$ are of constant
rank.

For point (ii), using Serre Duality and thinking of the spaces $H^p(X,\Omega_X^q)$ as the $(p,q)$-Dolbeault cohomology, 
we have a diagram
$$
\begin{array}[c]{ccccccccc}

\ldots & \stackrel{\wedge \omega}{\longrightarrow} & H^{d-p}(X,\Omega_X^{i-1})& \stackrel{\wedge \omega}{\longrightarrow} & H^{d-p}(X,\Omega_X^i)  & \stackrel{\wedge \omega}{\longrightarrow} & H^{d-p}(X,\Omega_X^{i+1}) & \stackrel{\wedge \omega}{\longrightarrow} &\ldots\\
       &			     & \downarrow           &                             &  	\downarrow     &                             &     \downarrow         &                             &      \\    		
\ldots & \longrightarrow                 & H^p(X,\Omega_X^{d-i+1})^{\vee} & \longrightarrow    & H^p(X,\Omega_X^{d-i})^{\vee} & \longrightarrow   & H^p(X,\Omega_X^{d-i-i})^{\vee} & \longrightarrow    &  \ldots

\end{array}
$$
where the bottom complex is the dual complex of \eqref{fiber}. 
This diagram commutes up to sign and hence, if $m>d-p$, 
the upper complex (and therefore also the bottom one) is exact at the 
first $m-d+p$ steps from the left. Finally dualizing again the bottom complex we have that \eqref{fiber} is exact at the first $m-d+p$ steps 
from the right. 

The case $m=\infty$ follows as the complexes \eqref{fiber} are now everywhere exact.
\end{proof}

In the special case when the zero-set of every non-zero holomorphic $1$-form consists of a finite number of 
points, i.e. when $m(X)=\ddim X$, Proposition \ref{exact} implies that the complex $\underline{\mathbf{L}}_X^p$
is everywhere exact except at most at one step. This allows us to give a bound on the Euler characteristic of the bundles of 
holomorphic $p$-forms in the case $q(X)>\ddim X$. We recall from the introduction that examples of manifolds with $m(X)=\ddim X$ are 
provided by smooth subvarieties of an 
abelian varieties having ample normal bundle.
Before stating the bounds, we prove a simple Lemma which will be useful in the sequel.

\begin{lemma}\label{lem}
Let $e\geq 2$, $t\geq 1$, $q\geq 1$ and $a$ be integers. For $i=1,\ldots ,e+1$ and $s=1,\ldots ,t$ let $V_i$ and $Z_s$ 
be complex vector spaces of positive dimension. 
\begin{enumerate}
\item If a complex of locally free sheaves on 
$\pp=\pp^{q-1}$ of length $e+1$ of the form
\begin{equation}\label{lemseq}
0\longrightarrow V_{e+1}\otimes \sO_{\pp}(-a)\longrightarrow V_e\otimes \sO_{\pp}(-a+1)\longrightarrow \ldots \longrightarrow V_1\otimes \sO_{\pp}(-a+e)\longrightarrow 0
\end{equation}
is exact, then $q\leq e$.\\
\item Let $k_s\geq -a+e$ be integers. If a complex of locally free sheaves on $\pp=\pp^{q-1}$ of length $e+2$ of the form
\begin{equation*}\label{lemseq2}
0\longrightarrow V_{e+1}\otimes \sO_{\pp}(-a)\longrightarrow V_e\otimes \sO_{\pp}(-a+1)\longrightarrow \ldots \longrightarrow V_1\otimes \sO_{\pp}(-a+e)\longrightarrow
\bigoplus_{s=1}^t (Z_s\otimes \sO_{\pp}(k_s))\longrightarrow 0
\end{equation*}
is exact, then $q\leq e+1$.
\end{enumerate}
\end{lemma}
\begin{proof}
If $q=1$ then clearly $q\leq e$, therefore we can assume $q>1$. If $e=2$ then $q=2$, since line bundles on projective spaces have no 
intermediate cohomology and so we can suppose $e>2$. After having twisted the complex \eqref{lemseq} by $\sO_{\pp}(-e+a)$ we get the complex
$$0\longrightarrow V_{e+1}\otimes \sO_{\pp}(-e)\stackrel{f_1}{\longrightarrow} V_e\otimes \sO_{\pp}(-e+1)\longrightarrow\ldots$$ $$\ldots  
\longrightarrow V_4\otimes \sO_{\pp}(-3)\stackrel{f_{e-2}}{\longrightarrow}V_3\otimes \sO_{\pp}(-2)\longrightarrow V_2\otimes \sO_{\pp}(-1)
\longrightarrow V_1\otimes \sO_{\pp}\longrightarrow 0.$$
 Set $W_j={\rm coker}\;f_j$ for $j=1,\ldots,e-2$. If 
$q>e$, we would have that 
$H^{e-1-j}(\pp,W_j)\neq 0$ for every $j=1,\ldots ,e-2$ and hence that $H^{e-1}(\pp,V_{e+1}\otimes \sO_{\pp}(-e))\neq 0$. This 
yields a contradiction and then $q\leq e$.
To prove (ii) we can use the same argument used to prove (i).
% (ii). If $\ddim Z_s=0$ for all $s=1,\ldots ,t$, then we have an exact complex of length $\leq e+1$ and hence by point (i) we have $q\leq e$.
% Hence we can assume $\ddim Z_s>0$ for some index $s$, but also this case follows easily arguing as in point (i). 
\end{proof}

\begin{cor}\label{cor}
Let $X$ be a compact K\"{a}hler manifold of dimension $d\geq 2$ and irregularity $q(X)>d$. If $m(X)=d$ 
then  
$$(-1)^{d-1}\chi(\Omega_X^1)\geq 2,$$ and
$$(-1)^{d-p}\chi(\Omega_X^p)\geq 1$$ for any $p=2,\ldots ,d-2$.
\end{cor}

\begin{proof}
To begin with, 
we note that $h^d(X,\Omega_X^p)\neq 0$ so that the complex $\underline{\llll}_X^p$ is non-zero. 
By Proposition \ref{exact} (ii), the assumption $m(X)=d$ implies that the non-zero complex 
$\underline{\mathbf{L}}_X^{d}$ 
is exact at the first $d$ steps from the right. If we had $h^d(X,\Omega_X^p)=h^p(X,\omega_X)=0$, then
$\underline{\mathbf{L}}_X^{d}$ would induce an exact complex of length $\leq d$ whose terms are sums of line bundles all of the
same degree, and by Lemma \ref{lem} we would have a contradiction. 

By Proposition \ref{exact} the complex $\underline{\mathbf{L}}_X^{p}$ is exact at the first $d-p$ steps 
from the left and at the first $p$ steps from the right. Therefore we get two exact sequences: 
$$0\longrightarrow \sO_{\pp}(-d)\otimes H^0(X,\Omega_X^p)\longrightarrow \ldots \longrightarrow \sO_{\pp}(-p-1)\otimes H^{d-p-1}(X,\Omega_X^p)
\stackrel{f}{\longrightarrow}$$ 
$$\stackrel{f}{\longrightarrow} \sO_{\pp}(-p)\otimes H^{d-p}(X,\Omega_X^p)\longrightarrow F\longrightarrow 0,$$ where the locally free sheaf
$F$ is the cokernel of the map
$f$, and 
$$0\longrightarrow G\longrightarrow \sO_{\pp}(-p)\otimes H^{d-p}(X,\Omega_X^p)\stackrel{g}{\longrightarrow}$$
$$\stackrel{g}{\longrightarrow} \sO_{\pp}(-p+1)\otimes H^{d-p+1}(X,\Omega_X^p)
\longrightarrow \ldots \longrightarrow \sO_{\pp}\otimes H^d(X,\Omega_X^p)
\longrightarrow 0,$$ where the locally free sheaf $G$ is the kernel of the map $g$. We also get 
an induced map of locally free sheaves $h:F \longrightarrow \sO_{\pp}(-p+1)\otimes H^{d-p+1}(X,\Omega_X^p)$,
which is of constant rank. 
Denoting by $E$ the kernel of $h$ we obtain a new exact sequence of locally free sheaves 
$$0\longrightarrow E\longrightarrow F\longrightarrow \sO_{\pp}(-p+1)\otimes H^{d-p+1}(X,\Omega_X^p)\longrightarrow \ldots \longrightarrow \sO_{\pp}\otimes
 H^d(X,\Omega_X^p)\longrightarrow 0$$ from which we can read ${\rm rank}\, E=(-1)^{d-p}\chi(\Omega_X^p)$.
If the locally free sheaf $E$ were the zero sheaf then the complex $\underline{\llll}_X^p$ would be an exact complex 
of length $\leq d+1$ whose terms are sum of 
line bundles all of the same degree, which is impossible by Lemma \ref{lem} (i). Thus 
$${\rm rank}\; E=(-1)^{d-p}\chi(\Omega_X^p)\geq 1.$$
If $p=d-1$ we can improve our bound. In this case the complex $\underline{\llll}_X^{d-1}$ is exact
at the first $d-1$ steps from the right, and hence we get a short exact sequence 
$$0\longrightarrow \sO_{\pp}(-d)\otimes H^0(X,\Omega_X^{d-1})\stackrel{h'}{\longrightarrow} G\longrightarrow E'\longrightarrow 0,$$
where $E'$ is the cokernel of the map $h'$. The locally free sheaf $E'$ is non-zero again by Lemma \ref{lem}.
If the rank of $E'$ were one, then $E'$ would be a line bundle, $E'=\sO_{\pp}(k)$ for some integer $k$, and 
$G\in \mbox{Ext}^1(\sO_{\pp}(k),\sO_{\pp}(-d)\otimes H^0(X,\Omega_X^{d-1}))=H^1(\pp,\sO_{\pp}(d-k)\otimes H^0(X,\Omega_X^{d-1})^{\vee})=0$.
Hence $G$ would split as a sum of line bundles and by Lemma \ref{lem} (ii) this is not possible. 
Therefore $${\rm rank}\; E'=(-1)^{d-1}\chi(\Omega_X^1)\geq 2.$$
\end{proof}

% \begin{ex}
% In this example we compute the Euler characteristic of the bundle of holomorphic $1$-forms for a smooth principal polarization 
% of an abelian variety, so that we can compare it with Corollary \ref{cor}. Let $\Theta$ be a smooth 
% principal polarization (i.e. $h^0(A,\sO(\Theta))=1$) of a $g$-dimensional abelian variety $A$, where $g\geq 3$. Using the 
% fact that the tangent bundle of an abelian variety is trivial and the conormal exact sequence 
% $$0\longrightarrow \sO_{\Theta}(-\Theta)\longrightarrow \bigoplus_g \sO_{\Theta}\longrightarrow \Omega_{\Theta}^1\longrightarrow 0$$
% we have that 
% $$\chi(\Omega_{\Theta}^1)=g\cdot \chi(\sO_{\Theta})-\chi(\sO_{\Theta}(-\Theta))=g\cdot \chi(\sO_{\Theta})-(-1)^g\chi(\sO_{\Theta}(2\Theta))=(-1)^g(2^g-g-1).$$ 
% \end{ex}
% Following \cite{GL1} (Remark after the proof of Proposition 3.7), the cohomological support loci
% $$V^i(\Omega^j_X)=\{\alpha \in \mbox{Pic}^0(X) \;|\; H^i(X,\Omega^j_X \otimes \alpha)\neq 0\}\subset \mbox{Pic}^0(X)$$ 
% consist of finitely many points as long as $i+j<m(X)$. In particular
% for $m(X)=d$ the loci $V^i(\omega_X)$ consist of finitely many points for every $i>0$. It follows that if $q\geq d$ then the sheaf $\omega_X$ is a GV-sheaf (i.e. 
% $\mbox{codim} V^i(\omega_X)\geq i$ for 
% every $i>0$) and hence by Proposition 1.19 in \cite{PA} $X$ is of maximal Albanese dimension (i.e. the Albanese map is generically finite 
% onto its image). For generalities and further developments about GV-sheaves we refer to \cite{GV}, \cite{GSV} and \cite{PA}.

\section{Inequalities for the Hodge Numbers}

After having studied the exactness of the complex \eqref{BGG-intr} we can derive inequalities for the Hodge numbers by using well-known results 
for locally free sheaves on projective spaces: the Evans-Griffith Theorem and the non negativity of the Chern classes for 
globally generated locally free sheaves. 
% \begin{theorem}[Evans-Griffith Theorem]
%  Let $F$ be a vector bundle of rank $f\geq 2$ on the projective space $\pp ^n$ satisfying 
% $$H^i(\pp ^n,F(k))=0$$ for all $1\leq i\leq f-1$ and all $k\in \zz$. Then $F$ splits as a sum of line bundles.
% \end{theorem}

% Inequalities hold for irregular compact K\"{a}hler manifolds $X$ with
% $m(X)={\rm min}\{{\rm codim}\;Z(\omega)\;|\; 0\neq \omega\in H^0(X,\Omega_X^1)\geq 2$.

Throughout this section we fix integers $d\geq 1$, $q\geq 1$, $0\leq p\leq d$ and $0< m\leq d$. We denote by $h^{p,q}=h^{p,q}(X)=
\ddim H^q(X,\Omega_X^p)$ the Hodge numbers of $X$ and by $q=q(X)$ the irregularity of $X$.

Before stating the results we need to introduce some notation.
If $d-p<m\leq d$, for $1\leq i\leq q-1$ we define $\gamma_i(X,\Omega_X^p)$ to be the coefficient of $t^i$ in the formal power series:
$$\gamma(X,\Omega_X^p;t)\stackrel{def}{=} \prod_{j=1}^{m-d+p}(1-jt)^{(-1)^j h^{p,2d-m-p+j}}\in \zz[[t]].$$
If $p<m\leq d$, for $1\leq i\leq q-1$ we define $\delta_i(X,\Omega_X^p)$ to be the coefficient of $t^i$ in the formal power series:
$$\delta(X,\Omega_X^p;t)\stackrel{def}{=}\prod_{j=1}^{m-p} (1-jt)^{(-1)^jh^{p,m-p-j}}\in \zz[[t]].$$
If $m=\infty$, for $i=1,\ldots q-1$ we define $\varepsilon_i(X,\Omega_X^p)$ to be the coefficient 
of $t^i$ in the formal power series:
$$\varepsilon(X,\Omega_X^p;t)\stackrel{def}{=}\prod_{j=1}^{d} (1-jt)^{(-1)^jh^{p,d-j}}\in \mathbf{Z}[[t]].$$
Also consider the following pieces of the Euler characteristic of the bundle $\Omega_X^p$. If $d-p<m \leq d$ define 
$$\chi^{\geq 2d-m-p}(\Omega_X^p) \stackrel{def}{=}\sum_{j=2d-m-p}^d (-1)^{2d-m-p+j}h^{p,j}$$ and if $p<m\leq d$ define
$$\chi^{\leq m-p}(\Omega_X^p) \stackrel{def}{=}\sum_{j=0}^{m-p}(-1)^{m-p+j}h^{p,j}.$$

\begin{theorem}\label{schur}
Let $X$ be a compact K\"{a}hler manifold of dimension $d$ and irregularity $q\geq 1$. Let 
$m=m(X)={\rm min}\{\codim Z(\omega)\;|\; 0\neq \omega \in H^0(X,\Omega_X^1)\}$ and let $0\leq p\leq d$ be an integer. 
\begin{enumerate}
 \item If $d-p<m\leq d$ then any Schur polynomial of weight $\leq q-1$ in the $\gamma_i(X,\Omega_X^p)$ is non-negative. 
In particular $$\gamma_i(X,\Omega_X^p)\geq 0$$ for every
$1\leq i\leq q-1$. Moreover, if $i$ is an index with $\chi^{\geq 2d-m-p}(\Omega_X^p)<i<q$, then $\gamma_i(X,\Omega_X^p)=0.$\\ 
 \item If $p<m\leq d$ then any Schur polynomial of weight $\leq q-1$ in the $\delta_i(X,\Omega_X^p)$ is non-negative. 
In particular $$\delta_i(X,\Omega_X^p)\geq 0$$ for every
$1\leq i\leq q-1$. Moreover, if $i$ is an index with $\chi^{\leq m-p}(\Omega_X^p)<i<q$, then $\delta_i(X,\Omega_X^p)=0$.\\
\item If $m=\infty$ then $$\varepsilon_i(X,\Omega_X^p)=0$$ for every $i=1,\ldots, q-1$.
\end{enumerate}
\end{theorem}
\begin{proof}
If $m>d-p$ then by Proposition \ref{exact} (ii) the complex $\underline{\llll}_X^p$ is exact at the first $m-d+p$ 
steps from the right, and hence we get the exact sequence 
\begin{equation}\label{ker}
 0\longrightarrow G\longrightarrow \sO_{\pp}(d-m-p)\otimes H^{2d-m-p}(X,\Omega_X^p)\stackrel{g}{\longrightarrow}
 \end{equation}
 \begin{equation*}
 \stackrel{g}{\longrightarrow} \sO_{\pp}(d-m-p+1)\otimes H^{2d-m-p+1}(X,\Omega_X^p)
 \longrightarrow \ldots \longrightarrow \sO_{\pp}\otimes H^d(X,\Omega_X^p)\longrightarrow 0,
 \end{equation*}
where $G$ is the kernel of the map $g$.
Tensoring \eqref{ker} by $\sO_{\pp}(d-m-p)$ and then dualizing it, 
we note that the polynomial $\gamma(X,\Omega_X^p;t)$ is the Chern polynomial of the locally free sheaf $G^{\vee}(m-d+p)$. 
Then its Chern classes $c_i(G^{\vee}(m-d+p))$, as well as 
the Schur polynomials in these, are non-negative since $G^{\vee}(m-d+p)$ is globally generated. 
In particular we get
$$\gamma_i(X,\Omega_X^p)={\rm deg}\; c_i(G^{\vee}(m-d+p))\geq 0.$$
The last statement of (i) follows from the fact that $c_i(G)=0$ for $i>{\rm rank}\;G=\chi^{\geq 2d-m-p}(\Omega_X^p)$.

The proof of (ii) is analogous to the proof of the previous point. If $m>p$ then by Proposition \ref{exact} (i) the 
complex $\underline{\llll}_X^p$ is exact at the first $m-p$ steps from the left and induces the following exact complex 
\begin{equation}\label{coker}
 0\longrightarrow \sO_{\pp}(-d)\otimes H^0(X,\Omega_X^p)\longrightarrow \ldots \longrightarrow \sO_{\pp}(-d+m-p-1)\otimes H^{m-p-1}(X,\Omega_X^p)
 \stackrel{f}{\longrightarrow} 
 \end{equation}
 \begin{equation*}
 \stackrel{f}{\longrightarrow} \sO_{\pp}(-d+m-p)\otimes H^{m-p}(X,\Omega_X^p)\longrightarrow F\longrightarrow 0
 \end{equation*}
where $F$ is the cokernel of the map $f$.
Tensoring \eqref{coker} by $\sO_{\pp}(d-m+p)$ we get
that the locally free sheaf $F(d-m+p)$ is globally generated and moreover that its Chern polynomial is the polynomial 
$\delta(X,\Omega_X^p;t)$. Now we conclude as in (i).

If $m=\infty$ then the complex $\underline{\mathbf{L}}_X^p$ is everywhere exact and the polynomial $\varepsilon(X,\Omega_X^p;t)$ is 
just the Chern polynomial of the zero sheaf. Thus its Chern classes satisfy $\varepsilon_i(X,\Omega_X^p)=0$, for every $i=1,\ldots ,q-1$.
\end{proof}

Under the assumption of Theorem \ref{schur} we also have 
\begin{theorem}\label{schur2}
 \begin{enumerate}
 \item If $d-p<m\leq d$ and $q>{\rm max}\{m-d+p,d-p-1\}$ then 
$$\chi^{\geq 2d-m-p}(\Omega_X^p)\geq q+d-m-p$$ and $$h^{d-p,1}\geq h^{d-p,0}+q-1.$$\\

\item If $p<m\leq d$ and $q>{\rm max}\{m-p,p-1\}$ then 
$$\chi^{\leq m-p}(\Omega_X^p) \geq q-m+p$$ and 
$$h^{p,1}\geq h^{p,0}+q-1.$$\\
% \item In particular if $m(X)=d$ and $q\geq d$ then
% $$\chi (\omega_X)\geq q-d.$$
 \end{enumerate}
\end{theorem}
\begin{proof}
(i). By Proposition \ref{exact} the complex $\underline{\mathbf{L}}_X^p$ is 
exact at the first $m-d+p$ steps from the right. Since $q>d-p-1$ we can prove, with an argument similar to the one used in 
Corollary \ref{cor}, that $h^d(X,\Omega_X^p)$ is non-zero and hence that 
the complex $\underline{\mathbf{L}}_X^p$ is non-zero as well (this observation allow us to use Lemma \ref{lem}). 
If $q=m-d+p+1$ then it is enough to prove that the rank of the locally free sheaf $G$ in \eqref{ker} 
is at least one. By Lemma \ref{lem} the locally free sheaf $G$ is neither zero nor 
splits as a sum of line bundles. 
Then by the Evans-Griffith Theorem (see \cite{LA} p. 92) $${\rm rank}\,G=\chi^{\geq 2d-m-p}(\Omega_X^p)\geq q+d-m-p.$$
For the inequality $h^{d-p,1}\geq h^{d-p,0}+q-1$, it is enough to note that 
the complex $\underline{\mathbf{L}}_X^p$ induces a surjection 
$H^{d-1}(X,\Omega_X^p)\otimes \sO_{\pp}\longrightarrow H^d(X,\Omega_X^p)\otimes \sO_{\pp}(1)\longrightarrow 0$
and, since $h^d(X,\Omega_X^p)\neq 0$, we can conclude thanks to Example 7.2.2 in \cite{LA}. 

(ii). The hypothesis $p< m\leq d$ implies that the complex $\underline{\mathbf{L}}_X^p$ is exact at the first $m-p$ steps 
from the left. 
Since $q>p-1$ we have that $h^0(X,\Omega_X^p)=h^d(X,\Omega_X^{d-p})\neq0$ as in (i), and therefore that  
the complex $\underline{\mathbf{L}}_X^p$ is non-zero as well. After having noted that 
${\rm rank}\; F=\chi^{\leq m-p}(\Omega_X^p)$ we can argue as in the previous point.
% (iii). The inequality $\chi(\omega_X)\geq q-d$ follows trivially by (i) for $q>d$. For the case $q=d$ the statement reduces to 
% $\chi(\omega_X)\geq 0$ which is a consequence of the Nakano Type Generic Vanishing Theorem (cf. \cite{GL1} Theorem 3.1). 

% There is also another proof of $\chi(\omega_X)\geq q-d$ and it goes as follows. It was observed in \ref{GL1} that for such manifolds 
% the cohomological 
% support loci $V^i(\omega_X)=\{\alpha \in \mbox{Pic}^0(X)\;|\;H^i(X,\omega_X\otimes \alpha)\neq 0\}$ consists of a finite number of points
% for every $i>0$. In particular the sheaf $\omega_X$ is a GV-sheaf (i.e. $\mbox{codim}_{\mbox{Pic}^0(X)}\;V^i(\omega_X)\geq i$, for any
% $i>0$) and by Remark 1.5 of \cite{LP} $X$ is of maximal Albanese dimension.  
\end{proof}

\subsection {The case $m(X)=\ddim X$}
When $X$ is an irregular compact K\"{a}hler manifold with $m(X)=\ddim X$ further inequalities 
hold thanks to Catanese's work \cite{CAT}. 
Let ${\rm alb}_X:X\longrightarrow {\rm Alb}(X)$ be the Albanese map of $X$. We say that $X$ is of \emph{maximal Albanese 
dimension} if $\ddim {\rm  alb}_X(X)=\ddim X$. 
Following Catanese's terminology we say that $X$ is of \emph{Albanese general type} if $q(X)> \ddim X$ and if it is 
of maximal Albanese dimension. 
A \emph{higher irrational pencil} is a surjective map with connected fibers
$f:X\longrightarrow Y$ onto a normal lower dimensional variety $Y$ and such that any smooth model of $Y$ is of Albanese general type.
\begin{lemma}\label{pencil}
If $X$ is an irregular compact K\"{a}hler manifold with $m(X)=\ddim X$, then $X$ does not carry
any higher irrational pencils.
\end{lemma}
\begin{proof}
We proceed by contradiction. Suppose a higher irrational pencil $f:X\longrightarrow Y$ exists and let
${\rm alb}_Y:Y\longrightarrow {\rm Alb}(Y)$ be the Albanese map of $Y$, which is well defined since $Y$
is normal.
The map ${\rm alb}_Y$ is not surjective, hence 
following an idea 
contained in the proof of \cite{EL} Proposition 2.2, one can show that given a general point $y\in Y$ there exists a 
holomorphic $1$-form $\omega$ of ${\rm Alb}(Y)$ whose restriction to ${\rm alb}_Y(Y)$ vanishes at ${\rm alb}_Y(y)$.
Pulling back $\omega$ to $X$, we get a holomorphic $1$-form which vanishes
along some fibers of $f$ which are of positive dimension, this contradicting the hypothesis $m(X)=\ddim X$. 
The form $\omega$ can be constructed as follows.
Let $z$ be a smooth point of the Albanese image ${\rm alb}_{Y}(Y)\subset {\rm Alb}(Y)$. The coderivative 
map $T_z^* {\rm Alb}(Y)\longrightarrow T_z^* {\rm alb_{Y}(Y})$ is surjective with non trivial kernel. 
Then take $\omega$ to be the extension to a holomorphic $1$-form on ${\rm Alb}(Y)$ of any non-zero form belonging to this kernel.
\end{proof}

% \begin{proof}
%  We proceed by contradiction. Let 
% $f:X\longrightarrow Y$ be a higher irrational pencil, eventually by replacing $Y$ by a smooth model, we can suppose $Y$ smooth.
% Then if $m(Y)<\infty$, by pulling back a holomorphic $1$-form 
% with zeros, we would have a holomorphic $1$-form on $X$ whose zero set would contain some positive 
% dimensional fibers of $f$, this contradicts
% the assumption $m(X)=d$. Therefore it must be $m(Y)=\infty$, and since $Y$ is of Albanese general type, there exist at least $\ddim Y$ 
% holomorphic $1$-forms which never vanish, thus $\Omega_Y^1$ must be the trivial bundle of rank $\ddim Y$ and again we reach 
% a contradiction since the
% irregularity of $Y$ has to be strictly greater than its dimension. 
% \end{proof}
In the following Proposition we collect inequalities for Hodge numbers in the case $m(X)=\ddim X$, which  
will be used 
to give asymptotic bounds
for Hodge numbers in terms of the irregularity 
for manifolds of dimension three and four (\emph{cf.} Corollary \ref{3folds} and Corollary \ref{4folds}). 
These are essentially extracted from \cite{LP} Remark 3.3 and the references therein together with Lemma 3.3.
\begin{prop}\label{m=d}
 Let $X$ be an irregular compact K\"{a}hler manifold with $m(X)=\ddim X$. Then
\begin{eqnarray}\label{catanese}
h^{0,k}\geq k(q(X)-k)+1  
\end{eqnarray}
for any $k=0,\ldots ,\ddim X$.
If $\ddim X\geq 3$ then
\begin{eqnarray}\label{voisin}
 h^{0,2}\geq 4q(X)-10.
\end{eqnarray}
If $q(X)\geq \ddim X$, and for any value of $\ddim X$, then
\begin{eqnarray}\label{chi}
 \chi (\omega_X)\geq q(X)-\ddim X.
\end{eqnarray}
\end{prop}
\begin{proof}
One can show that if $\omega_1,\ldots ,\omega_{k}$ are linearly independent holomorphic $1$-forms  
then $\omega_1\wedge \ldots \wedge \omega_{k}\neq 0$. This can be done by 
induction on $k$ and by using the above Lemma \ref{pencil}, 
Theorem 1.14 and Lemma 2.20 in \cite{CAT}. 
This fact can be reformulated as saying that the natural maps
$$\phi_k:\bigwedge ^k H^1(X,\sO_X)\longrightarrow H^k(X,\sO_X)$$ are injective on primitive forms $\omega_1\wedge \ldots \wedge \omega_k$.
The set of such forms is the cone over the image of the Pl\"{u}cker embedding, i.e. over the Grassmannian $\gggg(k,H^1(X,\sO_X))$, 
and by comparing the dimensions we have the bounds.
For the inequality \eqref{voisin} we can apply the same argument used in \cite{LP} Remark 3.3 and for the inequality \eqref{chi} we note that 
when $q(X)\geq \ddim X$ then $X$ is of maximal Albanese dimension. In fact by \cite{GL1} Remark on p. 405, the cohomological support loci
$V^i(\omega_X)=\{\alpha\in {\rm Pic}^0(X)\; |\; H^i(X,\omega_X\otimes \alpha)\neq 0\}$ consist of at most a finite set of points and, 
by \cite{LP} Remark 1.4 or by \cite{BLNP} Proposition 2.7, $X$ is of maximal Albanese dimension. Now the inequality 
$\chi (\omega_X)\geq q(X)-\ddim X$ follows by \cite{GSV} Corollary 4.2 or by \cite{LP} Theorem 3.1. 
\end{proof}

\section{Examples and asymptotic bounds for threefolds and fourfolds}
In this section we list concrete inequalities coming from Theorems \ref{schur} and \ref{schur2} in the 
most interesting case $m(X)=\ddim X$, $q(X)\geq \ddim X$, and for $\ddim X=3,4,5$. Moreover 
for threefolds and fourfolds we list asymptotic bounds in terms of the irregularity $q(X)$ for all the Hodge numbers. We also point out 
that some of the inequalities are still valid for some values of $q(X)$ smaller than $\ddim X$ and 
that other inequalities hold for different values of $m(X)$. Set $q=q(X)$ for the irregularity and 
$h^{p,q}=h^{p,q}(X)$ for the Hodge numbers.

Let us start with Theorem \ref{schur}.
We get a first set of inequalities by imposing the conditions $\gamma_1(X,\Omega_X^k)\geq 0$ for $k=0,1,2$. 
Hence 
\begin{align*}
 h^{0,2}\geq 2q-3, \quad &  h^{1,1}\geq 2q  & \mbox{ for } \ddim X=3\\
 h^{1,2}\geq 2h^{1,1}-3q,\quad &  h^{1,2}\geq 2h^{0,2},  \quad   h^{0,3}\geq 2h^{0,2}-3q+4  & \mbox{ for } \ddim X=4\\
 h^{0,4}\geq 4q-3h^{0,2}+2h^{0,3}-5, \quad & h^{1,4}\geq 4h^{1,1}-3h^{1,2}+2h^{1,3},\quad h^{2,2}\geq 2h^{1,2}-3h^{0,2} & \mbox{ for } 
\ddim X=5
\end{align*}

Finer inequalities are obtained by solving $\gamma_2(X,\Omega_X^k)\geq 0$. Then for $\ddim X=3$ we have
\begin{align}\label{3foldsecond}
h^{0,2}\geq 2q-\frac{7}{2}+\frac{\sqrt{8q-23}}{2}, \quad  h^{1,1}\geq 2q-\frac{1}{2}+\frac{\sqrt{8q+1}}{2}
\end{align}
and for $\ddim X=4$ we get 
\begin{eqnarray}
h^{0,3} & \geq & 2h^{0,2}-3q+\frac{7}{2}+\frac{\sqrt{8h^{0,2}-24q+49}}{2}\label{4foldsecond03}\\
h^{1,2} & \geq & 2h^{1,1}-3q+\sqrt{4h^{1,1}-9q}\\
h^{1,2}& \geq & 2h^{0,2}-\frac{1}{2}+\frac{\sqrt{8h^{0,2}+1}}{2}\label{4foldsecond12}
\end{eqnarray}
where the quantity $4h^{1,1}-9q$ is non-negative by the second inequality of Theorem \ref{schur2} (i) and the 
quantity $8h^{0,2}-24q+49$ is non-negative by inequality \eqref{voisin}.
Finally for $\ddim X=5$ we get
\begin{eqnarray*}
h^{0,4}&\geq &4q-3h^{0,2}+2h^{0,3}-\frac{11}{2}+\frac{\sqrt{48q-24h^{0,2}+8h^{0,3}-79}}{2}\\
h^{1,4}&\geq &2h^{1,3}+4h^{1,1}-3h^{1,2}-\frac{1}{2}+\frac{\sqrt{48h^{1,1}-24h^{1,2}+8h^{1,3}+1}}{2}\\
h^{2,2}&\geq &2h^{1,2}-3h^{0,2}-\frac{1}{2}+\frac{\sqrt{8h^{1,2}-24h^{0,2}+1}}{2}\\
\end{eqnarray*}
which hold as long as the quantity under the square root are non-negative.

Applying Theorem \ref{schur2} with $m(X)=\ddim X$ and $q(X)\geq \ddim X$, we get for $\ddim X=3$
$$\chi(\omega_X)\geq q-3,\quad  h^{1,1}\geq 2q-1, \quad  h^{1,2}\geq h^{1,1}-2, \quad  h^{1,2}\geq h^{0,2}+q-1,$$
for $\ddim X=4$
\begin{eqnarray*}
\chi (\omega_X)\geq q-4, & h^{2,2}\geq h^{1,2}-h^{0,2}+q-2, & h^{1,3}\geq h^{1,2}-h^{1,1}+2q-3\\
h^{1,1}\geq 2q-1, & h^{1,2}\geq h^{2,0}+q-1, & h^{1,3}\geq h^{0,3}+q-1
\end{eqnarray*}
and for $\ddim X=5$
\begin{eqnarray*}
\chi (\omega_X)\geq q-5,& h^{1,4}\geq h^{1,3}-h^{1,2}+h^{1,1}-4, & h^{1,1}\geq 2q-1,\\
2h^{1,2}\geq h^{2,2}+h^{0,2}+q-3,& h^{1,2}\geq h^{0,2}+q-1,& h^{1,2}\geq h^{1,3}-h^{0,3}+q-2,\\
h^{1,3}\geq h^{0,3}+q-1, & h^{1,4}\geq h^{0,4}+q-1.
\end{eqnarray*}

We select the strongest of the inequalities above in dimension three and four in statements formulated asymptotically for simplicity:
\begin{cor}\label{3folds}
 Let $X$ be an irregular compact K\"{a}hler threefold with $m(X)=3$. Then asymptotically
$$h^{0,2}\succeq 4q,\quad h^{0,3}\succeq 4q, \quad h^{1,1}\succeq 2q+\sqrt{2q},\quad h^{1,2}\succeq 5q+\sqrt{2q}.$$
\end{cor}
\begin{proof}
The inequality \eqref{voisin} gives $h^{0,2}\succeq 4q$. The inequality $\chi(\omega_X)\geq q-3$ 
of Theorem \ref{schur2} implies the inequality
$h^{0,3}\geq h^{0,2}-2$ and therefore asymptotically $h^{0,3}\succeq 4q$.
The asymptotic bound for $h^{1,1}$ follows by \eqref{3foldsecond}. Since by Corollary \ref{cor} $\chi(\Omega_X^1)\geq 2$ 
we also get the bound for $h^{1,2}$.picasa for linux
\end{proof}
\begin{cor}\label{4folds}
Let $X$ be an irregular compact K\"{a}hler fourfold with $m(X)=4$. Then asymptotically
$$h^{0,2}\succeq 4q, \quad h^{0,3}\succeq 5q+\sqrt{2q},\quad h^{0,4}\succeq 4q$$
$$h^{1,1}\succeq 2q,\quad h^{1,2}\succeq 8q+2\sqrt{2q},\quad h^{1,3}\succeq 12q+3\sqrt{2q},\quad h^{2,2}\succeq 8q+4\sqrt{2q}.$$
\end{cor}
\begin{proof}
The asymptotic bounds for $h^{0,2},h^{0,3}$ and $h^{0,4}$ follow by \eqref{voisin}, \eqref{4foldsecond03}
and \eqref{catanese} respectively.
Using the second inequality of Theorem \ref{schur2} (i) we get $h^{1,1}\succeq 2q$, and by inequality \eqref{4foldsecond12}
we get the bound for $h^{1,2}$. By Corollary \ref{cor} we have $\chi(\Omega_X^1)\leq 2$ and 
$\chi(\Omega_X^2)\geq 1$ which imply the bounds for $h^{1,3}$ and $h^{2,2}$.
\end{proof}

\section{Regularity of the cohomology modules} 
In this section we give the proof of Theorem \ref{thm-intr1} from the Introduction. 
% We start by recalling the Bernstein-Gel'fand-Gel'fand (BGG) correspondence and by showing how it can be used to bound the regularity 
% of finitely generated graded modules over an exterior algebra.
% \subsection{BGG correspondence}
% Let $V$ be a $q$-dimensional complex vector space with basis $e_1,\ldots ,e_q$ and $W=V^{\vee}$ be the dual vector space with 
% dual basis $x_1,\ldots, x_q$. Let $E=\oplus_{i=0}^q \wedge ^i V$ be the exterior algebra over $V$
% and $S=\mbox{Sym}(W)$ be the symmetric algebra over $W$. We declare that elements of $W$ are of degree $1$ and elements of $V$ of degree $-1$.
% 
% Let $P=\oplus_i P_i$ be a finitely generated graded $E$-module. 
% The BGG correspondence associates to $P$ the complex $\mathbf{L}(P)$ of free graded modules over the symmetric 
% algebra $S$ in cohomological degree $0$ to $d$, given by
% $$\ldots \longrightarrow S\otimes_{\cc}P_{j+1}\longrightarrow S\otimes _{\cc}P_j \longrightarrow S\otimes _{\cc}P_{j-1}\longrightarrow \ldots$$
% with differentials $$s\otimes p\mapsto \sum _i x_is\otimes e_ip$$ 
% (cf. \cite{BGG}, \cite{EFS}, \cite{LP}, \cite{E} and \cite{CO}). We refer to this complex as the \emph{BGG complex associated to the module} $P$.
% 
% There exists a notion of regularity
% for graded finitely generated $E$-modules analogous to the Castelnuovo-Mumford regularity for graded 
% finitely generated $S$-modules. 

Let $V$ be a complex vector space and $E=\bigwedge^* V$ be the graded exterior algebra over $V$.
A finitely generated graded $E$-module $P=\bigoplus_{j\geq 0} P_j$ is called $m$\emph{-regular} if it is generated in degrees $0$ up
to $-m$, and if its minimal free resolution has at most $m+1$ linear strands. Equivalently, $P$ is $m$-regular if and only if 
${\rm Tor}_i^E(P,\cc)_{-i-j}=0$ for all $i\geq 0$ and all $j\geq m+1$. 
The \emph{dual} over $E$ of the module $P$ is defined to be the $E$-module
$Q=\hat{P}=\oplus_j P_{-j}^*$ (\emph{cf.} \cite{E}, \cite{EFS}, \cite{LP}).
% (so positive degrees are switched to negative ones and viceversa).

% The following Proposition is the bridge between the notion of regularity and the BGG correspondence (cf. \cite{E} Theorems 7.7-7.8, \cite{EFS} 
% Corollary 2.5 and \cite{LP} Proposition 2.2).
% \begin{prop}\label{BGGreg}
% Let $P$ be a finitely generated graded module over $E$ with no component of negative degree, say $P=\oplus_{i=0}^d P_i$. 
% Then $Q=\hat{P}$ is $m$-regular if and only if the complex $\mathbf{L}(P)$ is exact at the first $d-m$ steps from the left, 
% i.e. if and only if the sequence
% $$0\longrightarrow S\otimes_{\cc}P_d\longrightarrow S\otimes_{\cc}P_{d-1}\longrightarrow \ldots \longrightarrow S\otimes_{\cc}P_m$$ of $S$-modules 
% is exact.
% \end{prop}
% 
% \subsection{Regularity of cohomology modules of holomorphic p-forms}
We continue to denote by $X$ an irregular compact K\"{a}hler manifold of dimension $d$ and irregularity $q$. Set
$V=H^1(X,\sO_X)$ and $E=\bigwedge ^* V$.
In \cite{LP} the authors determined the regularity of the graded $E$-module $Q_X=\bigoplus_i H^i(X,\omega_X)$ by studying the 
exactness of the complex associated to its dual module $P_X=\bigoplus_i H^i(X,\sO_X)$ via the BGG correspondence.
Fore references on the BGG correspondence see \cite{BGG}, \cite{EFS} and Chapter 7B of \cite{E}.
By applying their same technique and by using Generic 
Vanishing Theorems for 
bundles of holomorphic $p$-forms (see \cite{GV} and \cite{CH}) 
we give a bound for the regularity of the $E$-modules $\bigoplus_i H^i(X,\Omega_X^p)$ for any $p$.

Fix an integer $p=0,\ldots ,d$. Via cup product consider the graded $E$-module
$$P_X^p=\bigoplus_i H^i(X,\Omega_X^{d-p})$$ where the graded piece $H^i(X,\Omega_X^{d-p})$ has degree $d-i$.
The dual module over $E$ of $P_X^p$ is the module
$$Q_X^p=\bigoplus_i H^i(X,\Omega_X^p)$$ 
where the graded piece $H^i(X,\Omega_X^p)$ has degree $-i$.
% The BGG complex associated to the module $P_X^p$ is the complex of $S$-modules in cohomological degree $0$ to $d$
% $$\mathbf{L}(P_X^p): \quad 0\longrightarrow S\otimes_{\cc} H^0(X,\Omega_X^p)\longrightarrow 
% S\otimes_{\cc} H^1(X,\Omega_X^p)\longrightarrow \ldots \longrightarrow S\otimes_{\cc} H^d(X,\Omega_X^p)\longrightarrow 0.$$
% Writing $\mathbf{P}=\pp^{q-1}=\mathbf{P}_{sub}(V)$, 
% the complex $\mathbf{L}(P_X^p)$ sheafifies to yield a linear complex
% $\underline{\mathbf{L}}(P_X^p)$ of locally free sheaves on $\mathbf{P}$:
% \begin{eqnarray}\label{BGG2}
% \underline{\mathbf{L}}(P_X^p): \, 0\longrightarrow \sO_{\pp}(-d)\otimes H^0(X,\Omega_X^p)\longrightarrow \sO_{\pp}(-d+1)\otimes H^1(X,\Omega_X^p)
% \longrightarrow \ldots
% \end{eqnarray}
% $$\ldots \longrightarrow \sO_{\pp}(-1)\otimes H^{d-1}(X,\Omega_X^p)\longrightarrow \sO_{\pp}\otimes H^d(X,\Omega_X^p)\longrightarrow 0.$$

Let $W=V^{*}$ be the dual vector spapicasa for linuxce of $V$ and $S={\rm Sym}(W)$ be the symmetric algebra over $W$. Also denote by $\pp=\pp_{\rm sub}(V)$ the 
projective space of dimension $q-1$ over $V$.
By applying the functor $\Gamma_*$ to the complex $\underline{\llll}_X^p$, we get the complex $\llll^p_X$ 
of $S$-graded modules in homological degree $0$ to $d$
% as $\Gamma_{*}(\underline{\llll}^p_X)$ (where the complex $\underline{\llll}_X^p$ is the complex \eqref{BGG} from Section I)
\begin{eqnarray*}
\llll_X^p:\quad  
0\longrightarrow S\otimes_{\cc} H^0(X,\Omega_X^p)\longrightarrow S\otimes_{\cc} H^1(X,\Omega_X^p)\longrightarrow 
\ldots \longrightarrow S\otimes_{\cc} H^d(X,\Omega_X^p)\longrightarrow 0.
\end{eqnarray*}
(see \cite{Ha} p. 118 for the definition of $\Gamma_*$).
% (recall that for a sheaf $\sF$ in $\pp^n$, we have $\Gamma_{*}(\sF):=\oplus_{m\geq 0}H^0(\pp^n,\sF(m))$).
% Mimicking the proof of Lemma 2.3 in \cite{LP} we can show that the BGG complex $\llll(P_X^p)$ is 
% isomorphic to the complex $\llll_X^p$. 
% 
% The modules $P_X^p$ and $Q_X^p$ are dual modules via Serre duality once we assign to $H^i(X,\Omega_X^p)$ degree $d-i$ and 
% $H^i(X,\Omega_X^{d-p})$ degree $-i$. Therefore by Proposition \ref{BGGreg} we can apply
% the BGG correspondence to the module $Q_X^p$ to study the regularity of $P_X^p$. 
% At this point Theorem \ref{thm-intr1} follows from the following Proposition which bounds the exactness of the complexes $\llll_X^p$
% and $\underline{\llll}_X^p$ in terms of a different set of parameters with respect to the ones used in Proposition \ref{exact}.
To bound the regularity of the module $Q_X^p=\bigoplus_i H^i(X,\Omega_X^p)$ is enough to understand until which step the complex 
$\llll_X^{d-p}$ is exact. In fact, by following the proof of \cite{LP} Lemma 2.3, one can show that the complex $\llll_X^{d-p}$ is identified 
with the complex associated to the module $P_X^p=\bigoplus H^i(X,\Omega_X^{d-p})$ via the BGG correspondence 
and one can use Proposition 2.2 (\emph{cf. loc. cit.}) to bound the regularity of $Q_X^p$. 
At this point Theorem \ref{thm-intr1} 
follows by the previous discussion together with point (i) of the following Proposition.
\begin{prop}\label{alb}
Let $X$ be an irregular compact K\"{a}hler manifold of dimension $d$. 
Let $k$ be the dimension of the generic fiber of the Albanese map ${\rm alb}_X:X\longrightarrow {\rm Alb}(X)$ 
and the let $f$ be the maximal dimension of a fiber of ${\rm alb}_X$. Set $l={\rm max}\{k,f-1\}$. 
\begin{enumerate}
 \item If $d-p>l$ then the complexes $\llll_X^p$ and $\underline{\mathbf{L}}_X^p$ are exact in the first $d-p-l$ steps from the left.\\
 \item If $p>l$ then the complexes $\llll_X^p$ and $\underline{\mathbf{L}}_X^p$ are exact in the first $p-l$ steps from the right.
\end{enumerate}
\end{prop}
\begin{proof}
We follow \cite{LP} Proposition 1.1. 
Let $\mathbf{A}=\mbox{Spec(Sym}(W))$ be the affine space corresponding to $V$ viewed as an affine algebraic variety.
Consider the following complex $\sK_p$ of trivial locally free sheaves on $\mathbf{A}$:
$$0\longrightarrow \sO_{\mathbf{A}}\otimes H^0(X,\Omega_X^p)
\longrightarrow \sO_{\mathbf{A}}\otimes H^1(X,\Omega_X^p)\longrightarrow \ldots \longrightarrow 
\sO_{\mathbf{A}}\otimes H^d(X,\Omega_X^p)\longrightarrow 0,$$
with maps given at each point of $\mathbf{A}$ by cupping with the corresponding element of $V$. 
Since $\Gamma(\mathbf{A},\sO_{\mathbf{A}})=S$ one sees that 
\begin{eqnarray}\label{gamma}
\mathbf{L}_X^p=\Gamma(\mathbf{A},\sK_p)
\end{eqnarray}
i.e. $\mathbf{L}_X^p$ is obtained from $\sK_p$ by applying the global sections functor. 
On the other hand the complex $\underline{\llll}_X^p$ is 
obtained by $\llll_X^p$ via sheafification which is an exact functor, and hence its exactness is
implied by the exactness of the complex $\llll_X^p$, and consequently by the exactness of $\sK_p$
since the functor global sections is exact on affine spaces. Let 
$\mathbf{V}$ be the vector space $V$ considered as a complex manifold.
By GAGA, the exactness of $\sK_p$ 
is equivalent to the exactness of $\sK_p^{{\rm an}}$, where $\sK_p^{{\rm an}}$ is the complex
$$0\longrightarrow \sO_{\mathbf{V}}\otimes H^0(X,\Omega_X^p)\longrightarrow \sO_{\mathbf{V}}\otimes H^1(X,\Omega_X^p)\longrightarrow \ldots 
\longrightarrow \sO_{\mathbf{V}}\otimes H^d(X,\Omega_X^p)\longrightarrow 0$$
of analytic sheaves on $\mathbf{V}$.
If $d-p>l$ then it is enough to check the exactness of $\sK_p^{{\rm an}}$ at the first $d-p-l$ steps from the left, or equivalently 
the vanishing of the cohomologies $\sH^i(\sK_p^{{\rm an}})$ for any $i<d-p-l$. Since the 
differentials of the complex $\sK_p^{{\rm an}}$ scale linearly in radial directions through the origin, it is then enough 
to check the vanishing of its stalks at the origin $0$, i.e.
$$\sH^i(\sK_p^{{\rm an}})_0=0$$ for $i<d-p-l$. 

Let now $p_1:X\times {\rm Pic}^0(X)\longrightarrow X$ and 
$p_2:X\times {\rm Pic}^0(X)\longrightarrow {\rm Pic}^0(X)$ be the projections onto the first and second
factor respectively and let $\sP$ be a normalized Poincar\'{e} line bundle on $X\times {\rm Pic}^0(X)$. 
Then Theorem 6.2 in \cite{CH} gives an identification
$$\sH^i(\sK_p^{{\rm an}})_0 \cong R^ip_{2*}(p_1^{*}\Omega_X^p\otimes \sP)_0, $$ via the 
exponential map ${\rm exp}:V\longrightarrow {\rm Pic}^0(X)$.
At this point Theorem 5.11 (1) and Theorem 3.7 in \cite{GV} imply that $R^ip_{2*}(p_1^{*}\Omega_X^p\otimes \sP)=0$ for any $i<d-p-l$.

The proof of (ii) is analogous to the proof in Proposition \ref{exact}.
\end{proof}

\end{document}